\documentclass{article}
\usepackage{amsthm}
\usepackage{amsmath}
\usepackage{graphicx}
\newtheorem{theorem}{Theorem}[section]

\newtheorem{lemma}[theorem]{Lemma}
\theoremstyle{definition}
\newtheorem{definition}[theorem]{Definition}

\usepackage{arxiv}
\usepackage{xcolor}
\usepackage[utf8]{inputenc} % allow utf-8 input
\usepackage[T1]{fontenc}    % use 8-bit T1 fonts
\usepackage{hyperref}       % hyperlinks
\usepackage{url}            % simple URL typesetting
\usepackage{booktabs}       % professional-quality tables
\usepackage{amsfonts}       % blackboard math symbols
\usepackage{nicefrac}       % compact symbols for 1/2, etc.
\usepackage{microtype}      % microtypography
\usepackage{enumitem}
\usepackage{comment}
\numberwithin{equation}{section}

\newcommand{\R}{\mathbb{R}}
\newcommand{\C}{\mathbb{C}}

\newcommand{\s}{\textcolor{white}{.......}}

\newcommand{\mba}{\mathbf{a}}
\newcommand{\mbb}{\mathbf{b}}

\title{A Note On Convexity Inequalities Of Weighted Matrix Geometric Means}

\author{
Victoria M.~Chayes\\
Department of Mathematics\\
Rutgers University\\
Piscataway, NJ 08854 \\
\texttt{vc362@math.rutgers.edu} 
}

\begin{document}

\maketitle

\begin{abstract}
We offer a new proof of uniform convexity inequalities for the Finsler manifold of nonpositive curvature taken on the space of positive-semidefinite matrices with the weighted matrix geometric mean defining the geodesic between two points. Using the technique of log majorization, we are able to characterize that the equality cases of said equalities occur if and only if the matrices commute, and hence are the same as in $\ell^p$. 
\keywords{Matrix Geometric Mean \and Uniform Convexity \and $p$-Schatten Norms \and Log Majorization}
\end{abstract}

\section{Introduction}
\s We can consider the space of positive $n\times n$ matrix as a Reimannian metric with geodesic distance between $A,B\in \mathbf{P}_{n\times n}(\C)$ of \begin{equation}
\inf L(\gamma); \;\;\; L(\gamma):=\int_a^b ||\gamma(t)^{-1/2}\gamma'(t)\gamma(t)^{-1/2}||_2 dt,
\end{equation}
with $\gamma: [a,b]\rightarrow \mathbf{P}_{n\times n}(\C)$ being a smooth curve between $A$ and $B$ and $L(\gamma)$ its arc length.

\s This space was first introduced in Skovgaard \cite{skovgaard} for its applications in statistics, but is of general interest as it is a quintessential example of a metric space with non-positive curvature. The corresponding unit time geodesic distance \begin{equation}
\delta_2(A,B)=\inf\left\{\int_0^1||\gamma(t)^{-1/2}\gamma'(t)\gamma(t)^{-1/2}||_2 \; dt, \;\;\; \gamma(0)=A,\; \gamma(1)=B \right\}
\end{equation}
is reached uniquely by the path \begin{equation}\label{UG}
\gamma(t)=A^{\frac{1}{2}}\left(A^{-\frac{1}{2}}B^{\frac{1}{2}}A^{-\frac{1}{2}} \right)^tA^{\frac{1}{2}}
\end{equation}
with midpoint \begin{equation}
A\# B=A^{\frac{1}{2}}\left(A^{-\frac{1}{2}}B^{\frac{1}{2}}A^{-\frac{1}{2}}. \right)^{\frac{1}{2}}A^{\frac{1}{2}}.	
\end{equation}	

\s This midpoint is the `geometric mean' of matrices $A$ and $B$, introduced by Pusz and Woronowicz \cite{PuszWormean} as a way of generalizing $\sqrt{xy}$ to sesquilinear forms,  and the location along the geodesic for $t\in [0,1]$ is the `weighted geometric mean', which has been studied in great detail with respect to its relation to the Golden-Thompson inequality, quantum entropy, relative quantum entropy, and R\'{e}nyi divergences \cite{ANDO1994113} \cite{Araki1990} \cite{Chayes2020GT} \cite{Hiai2019}  \cite{HiaiPetz} . 

\s Distance in this metric is invariant under conjugation, as \begin{equation}
|||(X\gamma(t)X^\ast)^{-1/2}(X\gamma(t)X^\ast)'(X\gamma(t)X^\ast)^{-1/2}|||=|||\gamma(t)^{-1/2}\gamma'(t)\gamma(t)^{-1/2}|||
\end{equation}
for any untitarily invariant norm $|||\cdot|||$, including $||\cdot||_2$. When $A$ and $B$ commute, we have \begin{equation} 
\gamma(t)=A^{1-t}B^t
\end{equation}
and \begin{equation}
\delta_2(A,B)=||\log(A^{-\frac{1}{2}}B^{\frac{1}{2}}A^{-\frac{1}{2}} )||_2	=||\log(A)-\log(B)||_2;
\end{equation}
and in general, \begin{equation}\label{distineq}
\delta_2(A,B)=||\log(A^{-\frac{1}{2}}B^{\frac{1}{2}}A^{-\frac{1}{2}} )||_2	\geq||\log(A)-\log(B)||_2.
\end{equation}

\s However, instead of taking the Hilbert-Schmidt norm, arc length and respectively geodesic distance can be defined with the $p$-Schatten norm with $1<p<\infty$\footnote{{\footnotesize The $p=1$ case can be defined in the same manner, but the weighted geometric mean is no longer the unique geodesic.}} of \begin{equation}
L(\gamma):=\int_a^b ||\gamma(t)^{-1/2}\gamma'(t)\gamma(t)^{-1/2}||_p \; dt,
\end{equation}
now producing a Finsler manifold of non-positive curvature. Equation \ref{UG} is still the \textit{unique} geodesic between two matrices $A$ and $B$ \cite {Bhatia2003}, but now with distance \begin{equation}
\delta_p(A,B)=||\log(A^{-\frac{1}{2}}B^{\frac{1}{2}}A^{-\frac{1}{2}} )||_p.
\end{equation}
Note distance is still invariant under conjugation, and a number of properties of the geodesic are preserved. Full properties of this Finsler manifold are explored in \cite {Bhatia2003}, both using techniques involving derivatives of the exponential map, and using log majorization.

\s Two significant advances in work have been done recently: in \cite{Conde2009}, Conde looks at the generalized Finsler manifold (ie $1<p<\infty$) and uses Clarkson–McCarthy inequalities \begin{equation}
2\left(||A||_p^p+||B||_p^p\right)\leq ||A+B||_p^p+||A-B||_p^p\leq 2^{p-1}\left(||A||_p^p+||B||_p^p \right)
\end{equation}
for $2\leq p<\infty$ and reversing for $1\leq p\leq 2$ to derive $r$-uniform convexity inequalities as defined in \cite{Ball1994}; most significantly, Conde establishes \begin{align}
\delta_p(A\# B, C)^2&\leq \frac{1}{2}\delta_p(A,C)^2+\frac{1}{2}\delta_p(B,C)^2-\frac{p-1}{4}\delta_p(A,B)^2\qquad\qquad 1<p\leq 2\label{bc1}\\
\delta_p(A\# B, C)^p&\leq \frac{1}{2}\delta_p(A,C)^p+\frac{1}{2}\delta_p(B,C)^p-\frac{1}{2^p}\delta_p(A,B)^p\qquad\qquad \;\;\;\; \; 2\leq p \label{bc2}
\end{align}
Then in \cite{Bhatia2006}, Bhatia and Holbrook look to define a geometric mean for three positive matrices and analyze this in terms of convex hulls for the Reimannian metric ($p=2$), in particular using a different method of proof relying on log majorization and conjugation invariance of the distance. 

\s This paper applies an additional note to these results: using a the methodology of Bhatia and Holbrook, we can show that all inequalities \textit{must} be strict unless the matrices involve $\Gamma$-commute, reducing the equality cases of these inequalities to those in $\ell^p$ for both two-matrix and three-matrix inequalities. This comes from a fairly simple but vital theorem: the inequality of Equation \eqref{distineq} is an equality if and only if $A$ and $B$ commute. In Section 2 we introduce the technique log majorization to prove this theorem, and in Section 3 we use it prove the equality cases of Equations \eqref{bc1}-\eqref{bc2}.

\s We note that the equality case is of particular interest, because Equations \eqref{bc1}-\eqref{bc2} are used to establish coefficients such as the modulus of convexity of the space being considered. The coefficients in $\ell^p$ are known to be the best coefficients. Showing that inequalities are sharp except in commuting cases shows that for these matrix spaces, non-commutativity always makes things `worse'. This is notable, given that there are some inequalities (ie singular value rearrangement inequalities in \cite{Chayes2020}) where inequalities between positive-semidefinite matrices $A$ and $\{UBU^\ast \colon U \text{ unitary} \}$ are not minimized when $A$ and $UBU^\ast$ commute. 

\section{Log Majorization And Equality Cases}

\s Let $\mba,\mbb\in\R^n$ with components labeled in descending order $a_1\geq\dots\geq a_n$ and $b_1\geq\dots\geq b_n$. Then $\mathbf{b}$ weakly majorizes $\mathbf{a}$, written $\mathbf{a}\prec_{w} \bf{b}$, when \begin{equation}
\sum_{i=1}^ka_i\leq \sum_{i=1}^k b_i, \qquad 1\leq k \leq n
\end{equation}
and majorizes $\mba\prec\mbb$ when the final inequality is an equality. Weak log majorization $\mba\prec_{w(\log)}\mbb$ is similarly defined for non-negative vectors as \begin{equation}
\prod_{i=1}^ka_i\leq \prod_{i=1}^k b_i, \qquad 1\leq k \leq n
\end{equation}
with log majorization $\mba\prec_{(\log)}\mbb$ when the final inequality is an equality. Majorization applied to matrices $A$ and $B$ is understood to be applied to the vectors of their singular values respectively, and is a vital tool for proving inequalities with respect to any unitarily invariant norms.

\s An important relationship between majorization and log majorization was established in \cite{GTEqualityCases} (Lemma 2.2):
\begin{lemma}\label{lemmastrictlog}
Let $\mba,\mbb\in\R_+^n$ with $\mba_{(\log)} \mbb$. Suppose there exists a stricly convex function $\Phi : \R \rightarrow \R$ such that $\sum_{i=1}^n \Phi(a_i)=\sum_{i=1}^n \Phi(b_i)$. Then $\mathbf{a}=\Theta \mathbf{b}$ for some permutation matrix $\Theta$.
\end{lemma}

\s This was extended in \cite{chayes2021reverse}(Theorem 2.3) to general strictly convex functions: 
\begin{lemma}\label{lemmastrictconvex}
Let $\mba,\mbb\in\R^n$ with $\mba\prec\mbb$. Suppose there exists a stricly convex function $\Phi : \R \rightarrow \R$ such that $\sum_{i=1}^n \Phi(a_i)=\sum_{i=1}^n \Phi(b_i)$. Then $\mathbf{a}=\Theta \mathbf{b}$ for some permutation matrix $\Theta$.
\end{lemma}

\s This, along with some known majorization results on matrices with equality cases from \cite{GTEqualityCases}, gives us the tools we need to prove the equality conditions:

\begin{theorem}\label{gstrict}
Let $A, B \in \mathbf{P}_n$, and suppose $[A,B]\neq 0$ and $p>1$. Then \begin{equation}
\delta_p(A,B)> ||\log(A)-\log(B)||_p	
\end{equation}	
\end{theorem}
\begin{proof}
By \cite{GTEqualityCases} (Theorem 3.1), for any two Hermitian matrices $H,K$, the majorization inequality \begin{equation}\label{lm1}
\lambda(H+K)\prec \lambda\left(\log\left(e^{K/2}e^{H}e^{K/2} \right) \right)	
\end{equation}	
holds. Writing $H=\log(A)$, $K=-\log(B)$ for positive matrices $A,B$, we rewrite Equation \ref{lm1} as \begin{equation}
\lambda(\log(A)-\log(B))\prec \lambda\left(\log\left(B^{-1/2}AB^{-1/2} \right) \right)		
\end{equation}	
and as $x\mapsto |x|^p$ is strictly convex for $p>1$, then by Lemma \ref{lemmastrictconvex} the equality case of
\begin{equation}
\delta_p(A,B)= ||\log(A)-\log(B)||_p	
\end{equation}	
implies that $\lambda(\log(A)-\log(B))=\lambda\left(\log\left(B^{-1/2}AB^{-1/2} \right) \right)$. Then by \cite{GTEqualityCases} (Corollary 2.4), $H$ and $K$ and hence $A$ and $B$ must commute. 
\end{proof}

\section{Convexity Implications}

\s We will make use of the concept of $\Gamma$-commuting for equality cases as established in \cite{Bhatia2006}: 
\begin{definition}
Let  $\Gamma_X(A)=XAX^\ast$ for any $X\in M_{n\times n}(\C)$. Then  $A, B, C\in \mathbf{P}_{n\times n}(\C)$ $\Gamma$-commute if there exists some $X\in M_{n\times n}(\C)$ such that $\Gamma_X(A), \Gamma_X(B)$, and $\Gamma_X(C)$ all commute.
\end{definition}
with the following equivalent conditions:
\begin{lemma}
$A, B, C\in \mathbf{P}_{n\times n}(\C)$ $\Gamma$-commute if and only if $AB^{-1}C=CB^{-1}A$ if and only if $[A^{-\frac{1}{2}}BA^{-\frac{1}{2}},A^{-\frac{1}{2}}CA^{-\frac{1}{2}}]=0$.
\end{lemma}
Note that these conditions clearly imply the reduction to standard commuting when $C=I$; ie $A,B,I$ $\Gamma$-commute if and only if $[A,B]=0$.

\s We also introduce notation for the `exponential unit sphere' \begin{equation}
E_p=\left\{U \colon \;\; U\in \mathbf{P}_n, \;  \delta_p(U,I)=1      \right\},
\end{equation}
as elements of the unit sphere appear in the traditional convexity inequalities we want to explore; we also state the inequalities in their more traditional form in $\ell^p$. We now have all the tools we need to address the strictness of Inequalities \eqref{bc1}-\eqref{bc2}.

\begin{theorem}
Let $A, B \in \mathbf{P}_n$, and $1<p\leq 2$. Then for any $C \in \mathbf{P}_n$ we have \begin{equation}
\frac{\delta_p(A,C)^2+\delta_p(B,C)^2}{2}\geq \delta_p(A\# B, C)^2+\frac{p-1}{4}\delta_p(A,B)^2 .
\end{equation}
Letting $C=I$, then \begin{equation}\label{i1}
\frac{\delta_p(A,I)^2+\delta_p(B,I)^2}{2} \geq\delta_p(A\# B, I)^2+\frac{p-1}{4}\delta_p(A,B)^2 .
\end{equation}
In particular, for $A,B\in E_p$, we have \begin{equation}\label{2c2}
1-\delta_p(A\# B, I)\geq \frac{p-1}{8}\delta_p(A,B)^2.	
\end{equation}	
The constant $\frac{p-1}{4}$ of Equation \eqref{i1} (and consequently $\frac{p-1}{8}$ of Equation \eqref{2c2}) is ideal, but there is only equality when $p=2$ with $[A,B]=0$. In general, the inequalities are strict unless $A,B$, and $C$ $\Gamma$-commute or $A$ and $B$ commute respectively. 
\end{theorem}	

\begin{proof}
We use a very similar method to \cite{Bhatia2006}. Let $A,B,C\in  \mathbf{P}_n$. As geodesic distance is invariant to conjugation, letting $M=A\# B$, we can define $\tilde{A}=M^{-1/2}AM^{-1/2}$, $\tilde{B}=M^{-1/2}BM^{-1/2}$, $\tilde{C}=M^{-1/2}CM^{-1/2}$, and then noting that $\tilde{A}\#\tilde{B}=I$ and hence $\log(\tilde{B})=-\log(\tilde{A})$, we can write all of the following geodesic distance relationships
\begin{align}
\delta_p(M,C)&=\delta_p(\tilde{C}, I)=||\log(\tilde{C})||_p\label{li0} \\
\delta_p(A,B)&=\delta_p(\tilde{A}, \tilde{B}) =||2\log(\tilde{A})||_p\\
\delta_p(A,C)&=\delta_p(\tilde{A}, \tilde{C})\geq||\log(\tilde{A})-\log(\tilde{C})||_p \label{li1} \\
\delta_p(B,C)&=\delta_p(\tilde{B}, \tilde{C})\geq ||\log(\tilde{B})-\log(\tilde{C})||_p=||\log(\tilde{A})+\log(\tilde{C})||_p \label{li2}
\end{align}	
Applying the known 2-uniform convexity inequality for matrices \cite{Ball1994} \begin{equation}\label{2uc}
\frac{||X+Y||_p^2+||X-Y||_p^2}{2}\geq ||X||_p^2+(p-1)||Y||_p^2	\qquad \qquad 1\leq p \leq 2
\end{equation}	
we have \begin{equation}
\frac{||\log(\tilde{A})+\log(\tilde{C})||_p^2+||\log(\tilde{A})-\log(\tilde{C})||_p^2}{2}	\geq ||\log(\tilde{C})||_p^2+(p-1)||\log(\tilde{A})||_p^2
\end{equation}	
and hence \begin{equation}
\frac{\delta_p(A,C)^2+\delta_p(B,C)^2}{2}-\frac{p-1}{4}\delta_p(A,B)^2 \geq \delta_p(A\# B, C)^2\label{t1r1}
\end{equation}
Choosing $C=I$ and letting $A,B\in E_p$, we see \begin{equation}
1-\delta_p(A\# B, I)\geq \frac{p-1}{8}\delta_p(A,B)^2	\label{t1r2}	
\end{equation}	

\s To consider the sharpness of constancts and equality case, we note by Lemma \ref{gstrict}, the Inequality \eqref{li1} is strict unless $\tilde{A}$, $\tilde{B}$, and $\tilde{C}$ commute. This gives our $\Gamma$-commuting and commuting requirements for equality. Therefore the question is reduced to that of 2-uniform convexity in $\ell^p$, where the constant $(p-1)$ is ideal. As the 2-uniform convexity constant can be seen as a second order expansion of p-uniform convexity and the ideal constant from Hanner's inequality, there is no equality until $p=2$, when the inequality agrees with Hanner's inequality; then the only inequalities inolved in the expression are of Lines \eqref{li1} and \eqref{li2}, which by Lemma \ref{gstrict} are strict if and only if $[A,B]=0$. 
\end{proof}

\begin{theorem}
Let $A, B \in \mathbf{P}_n$, and $p\geq 2$. Then for any $C \in \mathbf{P}_n$ we have \begin{equation}
\frac{\delta_p(A,C)^p+\delta_p(B,C)^p}{2} \geq 2^{-p}\delta_p(A,B)^p+\delta_p(A\# B,C)^p
\end{equation}
Letting $C=I$, then \begin{equation}
\frac{\delta_p(A,I)^p+\delta_p(B,I)^p}{2} \geq 2^{-p}\delta_p(A,B)^p+\delta_p(A\# B,I)^p
\end{equation}
In particular, for $A,B\in E_p$, we have \begin{equation}
1-\delta_p(A\# B,I)^p\geq 2^{-p}\delta_p(A,B)^p
\end{equation}	
All the inequalities are strict unless $A,B$, and $C$ $\Gamma$-commute or $A$ and $B$ commute respectively. 
\end{theorem}	
\begin{proof}
Using the distance formulations and conjugation of Equations \eqref{li0}-\eqref{li2} and now the Clarkson–McCarthy inequalities, we have \begin{align}
\frac{\delta_p(A,C)^p+\delta_p(B,C)^p}{2}&\geq \frac{||\log(\tilde{A})+\log(\tilde{C})||_p^p+||\log(\tilde{A})-\log(\tilde{C})||_p^p}{2} \\
&\geq ||\log(\tilde{A})||_p^p+||\log(\tilde{C})||_p^p \\
&=2^{-p}\delta_p(A,B)^p+\delta_p(M,C)^p
\end{align}	
Once more, the inequalities of Equations \eqref{li1} and \eqref{li2} are strict unless $A,B$, and $C$ $\Gamma$-commute or $A$ and $B$ commute in the choice of $C=I$.  
\end{proof}

\s For completeness, we also add the following theorem addressing $p$-uniform convexity for $1<p\leq 2$, which does not appear in other literature: 

\begin{theorem}
Let $A, B \in \mathbf{P}_n$, and $1<p\leq 2$. Then for any $C \in \mathbf{P}_n$ we have \begin{equation}
\delta_p(A,C)^p+\delta_p(B,C)^p \geq \frac{\delta_p(A,B)^p+2^p\delta_p(A\# B,C)^p}{2}
\end{equation}
Letting $C=I$, then \begin{equation}
\delta_p(A,I)^p+\delta_p(B,I)^p \geq \frac{\delta_p(A,B)^p+2^p\delta_p(A\# B,I)^p}{2}
\end{equation}
In particular, for $A,B\in E_p$, we have \begin{equation}
1-2^{p-2}\delta_p(A\# B,I)^p \geq \frac{\delta_p(A,B)^p}{4}
\end{equation}	
All the inequalities are strict unless $A,B$, and $C$ $\Gamma$-commute or commute respectively. 
\end{theorem}
\begin{proof}
We once again use the distance formulations and conjugation of Equations \eqref{li0}-\eqref{li2} and the Clarkson–McCarthy inequalities to generate the inequality \begin{align}
\delta_p(A,C)^p+\delta_p(B,C)^p&\geq ||\log(\tilde{A})+\log(\tilde{C})||_p^p+||\log(\tilde{A})-\log(\tilde{C})||_p^p\label{wh0} \\
&\geq 2^{p-1}\left(||\log(\tilde{A})||_p^p+||\log(\tilde{C})||_p^p	\right)\label{wh} \\
&=\frac{\delta_p(A,B)^p}{2}+2^{p-1}\delta_p(M,C)^p
\end{align}	
Once more, Lemma \ref{gstrict} tells us the $\Gamma$-commuting and commuting requirements.

\s We note that the constants are not ideal, as for $1\leq p \leq \frac{4}{3}$, in Line \eqref{wh} the stronger inequality to use would be \begin{align}
\text{[Line \eqref{wh0}]} &\geq \left(||\log(\tilde{A})||_p+||\log(\tilde{C})||_p \right)^p	+\left|||\log(\tilde{A})||_p-||\log(\tilde{C})||_p \right|^p\\
&=\left(\frac{1}{2}\delta_p(A,B)+\delta_p(A\# B,C) \right)^p	+\left|\frac{1}{2}\delta_p(A,B)-\delta_p(A\# B,C)  \right|^p
\end{align}	
This is the matrix form of Hanner's inequality, and is conjectured to hold for the full regime of $1\leq p\leq 2$, but has not yet been proven outside of the given range except in the $p=\frac{3}{2}$ case \cite{heinavaara2022planes}.
\end{proof}

\bibliographystyle{spmpsci}   
\bibliography{references} 
\end{document}